\documentclass{article}

\usepackage[cmex10]{amsmath}
\usepackage{amssymb}
\usepackage{amsthm}
\usepackage{enumerate}
\usepackage{amscd}
\usepackage{tikz}

\theoremstyle{plain}
\newtheorem{theorem}{Theorem}
\newtheorem{lemma}[theorem]{Lemma}
\newtheorem{proposition}[theorem]{Proposition}

\newtheorem{problem}[theorem]{Problem}
\newtheorem{definition}[theorem]{Definition}

\theoremstyle{definition}

\usepackage{url}

\newcommand{\beq}{\begin{equation}}
\newcommand{\eeq}{\end{equation}}
\newcommand{\beqv}{\begin{equation*}}
\newcommand{\eeqv}{\end{equation*}}

\DeclareMathOperator{\hull}{Conv}

\newcommand{\R}{\mathbb{R}}
\newcommand{\F}{\mathbb{F}}
\newcommand{\longto}{\longrightarrow}

\newcommand{\abs}[1]{\lvert #1\rvert}
\newcommand{\norme}[1]{\lVert #1\rVert}
\newcommand{\scal}[2]{\langle#1,#2\rangle}

\let\subset\subseteq

\begin{document}

\title{On metric convexity, the discrete Hahn-Banach theorem, separating systems and sets of points forming only acute angles}

\author{Hugues Randriambololona\footnote{T\'el\'ecom ParisTech, 46 rue Barrault, 75013 Paris, France --- \texttt{randriam@enst.fr}}{ }\footnote{While writing this work the author was supported by ANR-14-CE25-0015 project \textsc{Gardio} and ANR-15-CE39-0013 project \textsc{Manta}.}}
\date{November 30, 2016}

\maketitle

\begin{flushright}
\emph{\`A G\'erard pour ses $2^{6}\!$ et $2^{6}\!+\!1$-\`emes anniversaires.}
\end{flushright}

\begin{abstract}
This text has three parts.

The first one is largely autobiographical, hence my use of the first person. There I recall how G\'erard Cohen influenced important parts of my research.

The second is of a more classic mathematical nature. I present a discrete analogue of the Hahn-Banach theorem, which serves as a basis for generalizing the notion of separating systems in the context of metric convexity.

The third one aims at building a bridge between two communities of researchers, those interested in separating systems, and those interested in a certain question in combinatorial geometry --- sets of points forming only acute angles --- who seem not to be aware of each other, while \emph{they are working on precisely the same problem}!

Of course, these three themes are closely intertwined.
\end{abstract}

\section{A young arithmetic geometer arrives at T\'el\'ecom}

I first met G\'erard by start of 2000 (or end of 1999?), while I was applying for a teaching assistant position at ENST, in the team of Mathematics for Computer Science and Networks.
G\'erard was head of the team, and it was Gilles Z\'emor, whose cryptography course I attended, who introduced us.
I still remember how warm and welcoming both of them were.
At this time I was in the middle of my Ph.D. thesis focused on Arakelov geometry and transcendental number theory.
I had only a very limited background on the team's main topics --- combinatorics and information theory --- so all I could provide was my willingness to work as best as I could with my future colleagues.
G\'erard seemed happy with that and he supported my application, which ended up in success.
For this, I already owed him a lot. But it was only the beginning.

The first years I shared office with Gilles.
While gradually getting involved in the teaching and research of the team, I was accorded a lot of autonomy and free time in order to complete my thesis.
G\'erard is a very witty person, always entertaining us with his jokes.
He especially liked to make fun of me at this time (and in fact, he still does so 15 years later!), because one part of my thesis was on heights of flat, generically zero dimensional, closed subschemes of projective spaces \cite{CRAS}.
How, he said, could a \emph{flat} thing have some \emph{height}?
Especially if it is not really a ``thing'', but only a sub-''thing'', and moreover, of dimension zero!

G\'erard quickly realized that the domain on which I could best interact with the team was that of algebraic geometry codes.
He suggested I read a few papers on the subject. Those on decoding were not my cup of tea, at least for a first contact.
But in Wei's theory of higher weights \cite{Wei1991}, and Tsfasman and Vladut's geometric view on them \cite{TV1995}, I could recognize some beauty.
At some point he also proposed me to look more closely at a new paper of Xing on frameproof codes \cite{Xing2002}.

Frameproof codes, also known as linear intersecting codes, were already studied by G\'erard quite some time ago \cite{CL1985}.
They are closely related to $(2,1)$-separating systems, whose definition I'll recall later in this text.
A simple probabilistic argument shows that the maximal cardinality $\kappa(n)$ of a $(2,1)$-separating system in the binary Hamming space $(\F_2)^n$ admits the lower bound
\beq
\label{kappa}
\kappa(n)\geq\left\lfloor\frac{1}{2}\left(\frac{2}{\sqrt{3}}\right)^n\right\rfloor
\eeq
which corresponds to an asymptotic binary rate of
\beq
\label{probbound}
\limsup_{n\to\infty}\frac{1}{n}\log_2\kappa(n)\;\geq\;1-\frac{1}{2}\log_23\approx0.207518.
\eeq
This probabilistic argument is non-constructive, but together with H.~G.~Schaathun, G\'erard observed that quite good constructions could be obtained by a concatenation procedure \cite{CS2003}.
The best candidates for this construction were: as fixed inner code, a certain variant of the Nordstrom-Robinson code, and as outer codes, an asymptotic family of linear intersecting codes over $\F_{121}$.
Using AG codes on the Tsfasman-Vladut-Zink bound, and a simple criterion for the intersection property based on the minimum distance, they were able to achieve a rate of $0.184503$.

But then, Xing's paper proposed a more refined criterion for the intersection property, specific to AG codes.  
Thus G\'erard asked me if I could help understand better what this was about,
and to start with, give a talk on the subject during a mini-workshop he was organizing at ENST.
At some point in his construction, Xing used the basic upper bound $2^{2g}$ on the number of $2$-torsion rational points of the Jacobian of the genus $g$ curve from which the code comes.
During my talk \cite{NTC}, someone (if I remember correctly, Fran\c{c}ois Morain) observed that this was the worst case indeed, but for one given curve, very often, one could hope for much better.

I immediately investigated the implications of this remark, which were tremendous.
Directly using Xing's result in the concatenation procedure gives separating systems of asymptotic rate $0.200877$.
However, if one could make the $2$-torsion group of the Jacobian negligible, this would improve the rate to
\beq
\label{conjbound}
\frac{3}{50}\log_211\approx0.207565
\eeq
better than the probabilistic bound!

When I presented this to G\'erard, he was very excited and encouraging,
although at the same time he couldn't help but make fun of me getting this ridiculously small improvement from \eqref{probbound} to \eqref{conjbound}, at the fifth digit only
(and he still does so, 13 years later, each time we mention the subject).
Anyway, this led me to state and start studying the following:
\begin{problem}
\label{2torsion}
For $q$ a prime square (e.g. $q=121$), construct a family of curves over $\F_q$, with an asymptotically optimal number of rational points,
and a negligible number of $2$-torsion rational points on their Jacobians. 
\end{problem}

Two approaches were possible, both quite difficult.
First, one could consider towers of asymptotically optimal curves given by explicit equations;
the task was then to control the $2$-torsion part of the class groups in the successive extensions.
Or, one could consider modular curves, which are known to be asymptotically optimal; the order of the corresponding modular Jacobians (hence a control on their $2$-torsion) would then be given by the determinant of a Hecke operator.
I found this second approach especially appealing.
Ideally, one could hope to get some Hecke operators with odd determinant, so there would be no $2$-torsion at all!
A natural strategy was to study the asymptotic repartition of eigenvalues of Hecke operators for the $2$-adic topology, precisely what Serre had done previously but for the ordinary topology \cite{Serre1997}.
With luck, this would only be an easy translation exercise.
And G\'erard was delighted that such elaborate number theory could be related to a fifth digit improvement on one of his favourite problems!

I spent quite some time and energy on this question, first alone, playing with modular symbols and with the Eichler-Selberg trace formula, but remained totally unsuccessfull.
As years passed, it became more a background question for me than a really active research topic.
Occasionally I discussed the subject with a few colleagues, among which my former thesis advisor, Bost, in Orsay, who suggested: why not ask Serre directly?
It took me more than a year, in 2006-2007, to prepare a letter on my computer, with carefully chosen words... which I finally never dared to send.

I remained stalled a couple more years, and could have remained so even much longer.
But then, again, it was G\'erard who gave the impetus to put things back in motion.
In 2009 he was invited in a workshop Xing was organizing in Zhangjiajie, China, and among other things he wanted to mention my funny fifth digit conjectural improvement in his talk.
Of course, he asked my permission for that. First I was a little bit reluctant.
I like to keep ideas for myself and make things public only when I have a clean, finished result.
However, G\'erard managed to convince me that sharing ideas and being more open is a good practice... especially there, considering I was not making any progress six years after.

In his presentation \cite{Cohen2009} he only makes a very brief mention.
He encouraged me to publish all the details by myself, including heuristics and partial results supporting the underlying conjectures.
I did so the year after \cite{ITW2010}.
Writing everything down had a rather unexpected, happy consequence.
It made me think again about Xing's result with a new eye, and doing so, suddenly I found a direct way to improve it, totally circumventing Problem~\ref{2torsion}.
The method was very natural, based on a geometric argument related to the theory of Weierstrass points --- details can be found in~\cite{IJM}.
The conjectural lower bound~\eqref{conjbound} is now a theorem!

As an epilogue to this story, a few months later I attended a talk given in Jussieu by Ronald Cramer.
His aim was to present a new notion he called ``torsion limit''.
I was very surprised to discover that this new notion was essentially the same that I had introduced in \cite{ITW2010}, in order to measure how close to a solution of Problem~\ref{2torsion} one could get!
There was an explanation: actually, Cramer was working with Xing.
Together with their student Cascudo, they were interested in certain problems to which the ideas of \cite{Xing2002} applied.
From this, it was very natural that they also considered Problem~\ref{2torsion}, independently.
Beside frameproof codes, the applications they had in mind were symmetric bilinear multiplication algorithms, and arithmetic secret sharing schemes.
I was happy to figure out that my new method, that allowed to bypass Problem~\ref{2torsion}, could also be used in the context of symmetric bilinear multiplication algorithms \cite{2D-G}\cite{ChCh}.
On the other hand, it does not seem to apply to arithmetic secret sharing systems, so a solution to Problem~\ref{2torsion} remains needed there.
Cascudo, Cramer and Xing have obtained important results in this direction \cite{CCX2011}.

\section{Convexity and separation}
\label{ConvSep}

Let $(E,d)$ be a metric space.
For $x,y\in E$ the segment $[x,y]$ is defined as the set of points $z\in E$
satisfying $d(x,y)=d(x,z)+d(z,y)$.
A subset $K$ of $E$ is said convex if, whenever $x,y\in K$,
then $[x,y]\subset K$.
Given a finite subset $S\subset E$, we define
its convex hull $\hull(S)$, as the smallest convex set
that contains $S$.

\begin{definition}
Given two integers $s,t$, a subset $C\subset E$ is said \emph{$(s,t)$-separating}
if for any subsets $S,T\subset C$ with $\abs{S}\leq s$, $\abs{T}\leq t$, and
$S\cap T=\emptyset$, one has also $\hull(S)\cap\hull(T)=\emptyset$.
\end{definition}

This allows to consider $(s,t)$-separation in any metric space
(for example, in graphs).
However for some spaces these notions are quite poorly behaved.
For instance
it may happen that segments are not convex: 

\medskip

\begin{center}
\begin{tikzpicture}
\draw (0,1.25) -- (2.5,0) -- (0,-1.25) -- (-2.5,0) -- (0,1.25) -- (0,-1.25);
\filldraw (0,0) circle (0.1);
\filldraw (-2.5,0) circle (0.1);
\filldraw (2.5,0) circle (0.1);
\filldraw (0,1.25) circle (0.1);
\filldraw (0,-1.25) circle (0.1);
\draw (-2.75,0.375) node {\large $x$};
\draw (2.75,0.375) node {\large $y$};
\end{tikzpicture}

\medskip

the bipartite graph $K_{3,2}$
\end{center}

\noindent Here, $[x,y]\varsubsetneq\hull(x,y)$.

\bigskip

On the other hand, there are spaces
in which these notions have very nice properties.
We first describe qualitatively what we expect these desirable properties
to be:

\begin{enumerate}[(P1)]
\item While in general $\hull(S)$ can always be described ``externally''
as the intersection of \emph{all} convex sets containing $S$, this may
not be very manageable. One would like that this intersection could
be taken over a smaller class of sets.
\item While $\hull(S)$ can always be constructed ``internally'',
starting with $S$, and saturating it under the operation
that, to a set $S'$, adjoins all the segments $[x,y]$ for $x,y\in S'$,
in general the number of iterations in this procedure could
not be bounded a priori. One would like to have such a bound (for example,
linear, or better, logarithmic in $\abs{S}$).
\item Last, one would like to have a direct or ``synthetic''
characterization of the individual elements of $\hull(S)$ in relation
to the elements of $S$.
\end{enumerate}

These three points are best illustrated in the following well-known example.

Let $(E,\scal{.}{.})$ be a Euclidean space.
A half-space in $E$ is (uniquely) defined as a subset of the form
\beqv
H_{u,\alpha}=\;\{x\in E \;|\;\scal{u}{x}\leq\alpha\}\;=\;l_u^{-1}(]-\infty,\alpha])
\eeqv
where $u$ is a unit vector in $E$, with associated linear form $l_u$, and $\alpha$ is a real.
 
Then, given a finite number of points $x_1,\dots,x_m\in E$, 
their convex hull $\hull(x_1,\dots,x_m)$ admits the following equivalent
descriptions:
\begin{enumerate}[(P1)]
\item 
$\hull(x_1,\dots,x_m)$ is the intersection of the
\emph{half-spaces} that contain
$x_1,\dots,x_m$.
\item $\hull(x_1)=\{x_1\}$,
and for $m\geq2$:
\beqv
\hull(x_1,\dots,x_m)=\bigcup_{x\in \hull(x_1,\dots,x_{m-1})}[x,x_m].
\eeqv
\item More directly: 
\beqv
\hull(x_1,\dots,x_m)=\left\{\lambda_1x_1+\cdots+\lambda_mx_m\;|\;\lambda_j\geq0,\sum\lambda_j=1\right\}.
\eeqv
\end{enumerate}

We consider now another example,
which is a perfect analogue of the preceding,
in a discrete setting.

Let $Q$ be a set of cardinality $\abs{Q}=q$
and let $E=Q^n$ be the
set of length~$n$ sequences over the alphabet $Q$. Equip $E$ with the
Hamming distance $d$.

Define a ``half-space'' in $E$ to be a subset of the form
\beqv
H_{i,\alpha}=\;\{x\in E \;|\; \pi_i(x)\neq\alpha\}\;=\;\pi_i^{-1}(Q\setminus\{\alpha\})
\eeqv
where $\pi_i:E\longto Q$ is projection
on the $i$-th coordinate ($1\leq i\leq n$), and
$\alpha\in Q$. Such a subset has cardinality
$\abs{H_{i,\alpha}}=(q-1)q^{n-1}=\frac{q-1}{q}\abs{E}$.
(In particular for $q=2$ a ``half-space'' really is a half-space.)

\begin{theorem}
\label{thHB}
Let $E=Q^n$, equipped with the Hamming distance.
Then, given a finite number of points $x_1,\dots,x_m\in E$, 
their convex hull $\hull(x_1,\dots,x_m)$ admits the following equivalent
descriptions:
\begin{enumerate}[(P1)]
\item 
$\hull(x_1,\dots,x_m)$ is the intersection of the ``half-spaces''
that contain
$x_1,\dots,x_m$.
\item  $\hull(x_1)=\{x_1\}$,
and for $m\geq2$:
\beqv
\hull(x_1,\dots,x_m)=\bigcup_{x\in \hull(x_1,\dots,x_{m-1})}[x,x_m].
\eeqv
\item More directly:
\beqv
\hull(x_1,\dots,x_m)=\left\{\:x\in E\;|\;\forall i\;\; \pi_i(x)\in\{\pi_i(x_1),\dots,\pi_i(x_m)\}\:\right\}.
\eeqv
\end{enumerate} 
\end{theorem}
\begin{proof}
We start with the following \emph{key observation}: for any $x,y\in E$,
\beq
\label{key}
[x,y]=\{\:z\in E\;|\;\forall i\;\; \pi_i(z)\in\{\pi_i(x),\pi_i(y)\}\:\}.
\eeq
This is shown by computing the contribution of each 
coordinate in the equality case $d(x,y)=d(x,z)+d(z,y)$
of the triangular inequality.

Temporarily define alternative ``convex hulls'' 
$\widetilde{\hull}_1$ $\widetilde{\hull}_2$ $\widetilde{\hull}_3$ as
given by descriptions (P1) (P2) (P3) respectively
(while ``$\hull$'' still denotes the original one).

First we prove
\beqv
\widetilde{\hull}_2(x_1,\dots,x_m)\subset\hull(x_1,\dots,x_m)
\eeqv
by induction on $m$. It clearly holds for $m=1$, so let $m\geq 2$ and suppose
it holds for $m-1$:
$\widetilde{\hull}_2(x_1,\dots,x_{m-1})\subset\hull(x_1,\dots,x_{m-1})$.
Then $\hull(x_1,\dots,x_m)$ being convex contains all the segments
$[x,x_m]$ for $x\in\widetilde{\hull}_2(x_1,\dots,x_{m-1})$,
hence it contains $\widetilde{\hull}_2(x_1,\dots,x_m)$ as claimed.

Now, in order to show
\beqv
\hull(x_1,\dots,x_m)\subset\widetilde{\hull}_3(x_1,\dots,x_m)
\eeqv
it suffices to show that $\widetilde{\hull}_3(x_1,\dots,x_m)$ is convex.
This will follow
from our \emph{key observation} \eqref{key}:
if $x,y\in\widetilde{\hull}_3(x_1,\dots,x_m)$ and $z\in[x,y]$,
then for any $i$, one has
$\pi_i(z)\in\{\pi_i(x),\pi_i(y)\}\subset\{\pi_i(x_1),\dots,\pi_i(x_m)\}$
hence $z\in\widetilde{\hull}_3(x_1,\dots,x_m)$, which is what
we wanted.

Now we prove
\beqv
\widetilde{\hull}_3(x_1,\dots,x_m)\subset\widetilde{\hull}_2(x_1,\dots,x_m)
\eeqv
by induction on $m$. It clearly holds for $m=1$, so let $m\geq 2$ and suppose
it holds for $m-1$.
Pick any $z\in\widetilde{\hull}_3(x_1,\dots,x_m)$
and define some $x\in\widetilde{\hull}_3(x_1,\dots,x_{m-1})$
coordinate by coordinate as follows:
\beqv
\pi_i(x)=
\begin{cases}
\pi_i(z) & \text{if $\pi_i(z)\in\{\pi_i(x_1),\dots,\pi_i(x_{m-1})\}$}\\
\pi_i(x_1) & \text{if $\pi_i(z)=\pi_i(x_m)\not\in\{\pi_i(x_1),\dots,\pi_i(x_{m-1})\}$.}
\end{cases}
\eeqv
Then for all $i$, $\pi_i(z)\in\{\pi_i(x),\pi_i(x_m)\}$,
hence $z\in[x,x_m]$,
so $z\in\widetilde{\hull}_2(x_1,\dots,x_m)$ and the conclusion follows.

Thus we have $\hull=\widetilde{\hull}_2=\widetilde{\hull}_3$, and
to conclude, write for any $S\subset E$:
\beqv
\begin{split}
\widetilde{\hull}_1(S)&=\bigcap_{1\leq i\leq n}\bigcap_{\;\alpha\not\in\pi_i(S)}\pi_i^{-1}(Q\setminus\{\alpha\})\\
&=\bigcap_{1\leq i\leq n}\pi_i^{-1}(\pi_i(S))\;=\;\widetilde{\hull}_3(S)
\end{split}
\eeqv
as wished.\end{proof}

Drawing the parallel with the Euclidean case,
it may be convenient to see (P1) as a \emph{discrete Hahn-Banach theorem}:
given any (Hamming-)convex set $K\subset Q^n$
and any $x\not\in K$, there is a coordinate $i$ that separates $x$
from $K$, in the sense that $\pi_i(x)\not\in\pi_i(K)$.

Both for Euclidean space and for Hamming space,
one can actually prove the following stronger version of (P2):
if $S$ is a finite set, written as a union $S=T\cup T'$,
then
\beq
\hull(S)=\bigcup_{x\in \hull(T),x'\in\hull(T')}[x,x'].
\eeq
This follows easily from (P3).
As a consequence, if one defines $K_0=S$ and inductively
$K_{i+1}=\bigcup_{x,x'\in K_i}[x,x']$,
then
$\hull(S)=K_{\lceil\log_2\abs{S}\rceil}$.

Observe that when $m=2$, description (P2) asserts that \emph{segments are
convex}.
In particular, a code $C\subset Q^n$
is $(2,1)$-separating when for any $x,y,z\in C$
with $z\neq x,y$, one has $z\not\in[x,y]$.
I first read this elegant definition in \cite{Korner1995},
and understanding it in a more general context was one of the motivations for this work.

\section{From separating systems to sets of points forming only acute angles, and vice-versa}

Separating systems (in Hamming spaces) have a very long history, and were rediscovered independently several times in various contexts.
When G\'erard introduced me to the subject, he was interested in applications to broadcast diffusion and traitor tracing.
In this context, given $m$ sequences $x_1,\dots,x_m$ in $E=Q^n$,
their convex hull as described by (P3) in Theorem~\ref{thHB} is often called the set of \emph{descendants} of $x_1,\dots,x_m$,
or the set of sequences that can be \emph{framed} by $x_1,\dots,x_m$.

Thanks to the Hahn-Banach property (P1), we see that a code $C\subset Q^n$ is $(s,t)$-separating
precisely when for any subsets $S,T\subset C$ with $\abs{S}\leq s$, $\abs{T}\leq t$, and
$S\cap T=\emptyset$, there is a coordinate $i$ such that
$\pi_i(S)\cap\pi_i(T)=\emptyset$.
In this situation we
say that $i$ is a \emph{separating coordinate} for $S$ and $T$,
or equivalently, that the codewords in $S$ and $T$ are \emph{separated}
at $i$.

Given a problem, G\'erard often enjoys devising modified versions.
Here, a natural generalization is the following:
%
\begin{definition}
Given integers $s,t$ and a real $\epsilon\geq0$, we say that a code
$C\subset Q^n$ is $\epsilon$-$(s,t)$-separating, if any disjoint $S,T\subset C$ with
$\abs{S}\leq s,\abs{T}\leq t$ admit a set $\Lambda\subset\{1,\dots,n\}$ of separating coordinates of cardinality $\abs{\Lambda}>\epsilon n$.
\end{definition}

The analogy between the Hamming and Euclidean Hahn-Banach theorems then suggests a very natural variant.
Say that two subsets $S,T$ of a Euclidean space $E$ are \emph{separated} by a vector $u$ in the unit sphere $S_E(1)$
if there is a real $\alpha$ such that one of $S$ or $T$ is included in the half-space $H_{u,\alpha}=\{x\in E \;|\;\scal{u}{x}\leq\alpha\}$, and the other is included in its complement. Then:
%
\begin{definition}
Given integers $s,t$ and
a real $\epsilon\geq0$, we say that a subset $C\subset E$ is
$\epsilon$-$(s,t)$-separating, if any disjoint $S,T\subset C$ with
$\abs{S}\leq s,\abs{T}\leq t$ admit a set $\Lambda\subset S_E(1)$ of separating vectors of measure $\mu(\Lambda)>\epsilon\mu(S_E(1))$.
\end{definition}

This definition might seem complicated, but it turns out it has a very nice geometric interpretation, best illustrated in the simplest case $(s,t)=(2,1)$:

\begin{proposition}
\label{angle}
A subset $C$ of a Euclidean space $E$ is $\epsilon$-$(2,1)$-separating
if and only if any three distinct $x,y,z$ in $C$ form an angle of measure $\widehat{xzy}<(1-\epsilon)\pi$.
\end{proposition}
\begin{proof}
A unit vector $u$ separates $z$ from the segment $[x,y]$ precisely when the affine hyperplane $z+u^\perp$ does not intersect $[x,y]$.
This condition depends only on the projection $p(u)$ of $u$ on the plane $(xyz)$.
Since $p(u)=0$ defines a set of measure $0$, we can assume $p(u)\neq0$.
Then we see that $u$ does \emph{not} separate $z$ from $[x,y]$ precisely when the line orthogonal to $p(u)$ in the plane lies in the cone from $z$ to $[x,y]$,
which defines a set of directions of relative measure $\frac{1}{\pi}\widehat{xzy}$.
\end{proof}

Actually, the problem of constructing a large $\epsilon$-$(2,1)$-separating system in the standard $n$-dimensional Euclidean space $E=\R^n$,
that is, constructing a large system of points subject to an upper constraint on the greatest angle between them,
has already been considered by Erd\"os and F\"uredi in \cite{ErdosFuredi}.

Of special importance is the case $\epsilon=\frac{1}{2}$: from Proposition~\ref{angle} we see that
a subset $C\subset \R^n$ is $\frac{1}{2}$-$(2,1)$-separating precisely when any three of its elements form an acute angle.
We let $\alpha(n)$ be the maximal possible cardinality of such a $C$.

So far, we rediscovered the Erd\"os-F\"uredi problem as a translation of the theory of separating systems from Hamming space to Euclidean space.
It turns out Erd\"os and F\"uredi went precisely in the opposite direction!
For this they focused on the special case of sets of points that are vertices of the unit cube:
\begin{lemma}[\cite{ErdosFuredi}]
Consider a subset $C\subset\{0,1\}^n\subset\R^n$ of vertices of the unit cube in the standard $n$-dimensional Euclidean space.
Then $C$ is Euclidean $\frac{1}{2}$-$(2,1)$-separating if and only if $C$ is Hamming $(2,1)$-separating.
\end{lemma}
\begin{proof}
For $x,y,z\in\{0,1\}^n$, the scalar product $\scal{x-z}{y-z}$ is equal to the number of Hamming separating coordinates between $\{x,y\}$ and $\{z\}$.
\end{proof}

As a consequence, we have
\beqv
\alpha(n)\geq\kappa(n)\geq\left\lfloor\frac{1}{2}\left(\frac{2}{\sqrt{3}}\right)^n\right\rfloor.
\eeqv

Actually, Erd\"os and F\"uredi did not use the term ``separating systems'' and did not seem aware of the literature on the subject.
And conversely, specialists in separating systems seem not to have noticed the paper \cite{ErdosFuredi}.
This led to the very unfortunate situation of the development of two disjoint series of works on the same topic:
\begin{itemize}
\item On the separating systems side, we can cite among other \cite{Renyi1961}\cite{FGU69}
\cite{Saga1978}\cite{FK84}\cite{Korner1995}\cite{CS2003}\cite{IJM} and the extensive survey \cite{SC2009}.
\item The Erd\"os-F\"uredi problem is considered in \cite{ErdosFuredi}\cite{Bevan2006}\cite{AB2009} and it is also mentioned in the classic books \cite{AlonSpencer}\cite{AignerZiegler}.
\end{itemize}

As an illustration of this phenomenon, we observe that K\"orner gives two reformulations of the notion of $(2,1)$-separation in \cite{Korner1995}.
We already mentioned the first one: a binary Hamming $(2,1)$-separating system is a set of binary sequences no three of which are on a line.
His second reformulation is in terms of set systems: a binary Hamming $(2,1)$-separating system is the same thing as a set system no three elements of which $A,B,C$ satisfy
\beqv
A\cap B\subset C\subset A\cup B.
\eeqv
Actually this formulation can already be found in \cite{ErdosFuredi}!
(Interestingly, although K\"orner does not refer to \cite{ErdosFuredi}, he cites and discusses the related notion of $2$-cover free families from \cite{EFF1982}.)

\vspace{\baselineskip}

What is the best lower bound on $\kappa(n)$ up to now?
Some improvements on \eqref{kappa} can be found in \cite{Bevan2006} and \cite{AB2009}, still relying on the probabilistic method. They read
\beqv
\kappa(n)\geq2\left\lfloor\frac{\sqrt{6}}{9}\left(\frac{2}{\sqrt{3}}\right)^n\right\rfloor
\eeqv
\beqv
\kappa(n)\geq\Omega\left(\left(\frac{2}{\sqrt{3}}\right)^n\sqrt{n}\right)
\eeqv
respectively, so they still have the same asymptotic exponent $1-\frac{1}{2}\log_23\approx0.207518$ as from \eqref{probbound}.

However, from \cite{IJM} we get
\beqv
\kappa(n)\geq 11^{\left(\frac{3}{50}-o(1)\right)n}
\eeqv
which is asymptotically better, with the exponent $\frac{3}{50}\log_211\approx0.207565$ of \eqref{conjbound}.
Moreover, this is obtained by an algebraic construction.

Actually, a bound $\kappa(n)\geq(\sqrt[4]{2}-o(1))^n$ is claimed in \cite{ErdosFuredi} without proof. If true, this would be even better.
Unfortunately, I was unable to reproduce it. I also asked F\"uredi if he was able to retrieve it, but got no answer yet.
The proof, if it ever existed, seems lost.

In the first months of 2016 G\'erard received an email from N.~Sloane, who was interested in collecting the first values of $\kappa(n)$ for the On-Line Encyclopedia of Integer Sequences.
He knew of the probabilistic results from \cite{Bevan2006} and \cite{AB2009}, and he was enquiring G\'erard about possible algebraic constructions from coding theory.
When G\'erard brought this as a coffee discussion with the team,
I explained how this was related to separating systems and proposed to write back to Sloane personally.
The story is now included in \cite{OEIS}.

\vspace{\baselineskip}

As a conclusion, I presented how the Hahn-Banach theorem allows to translate the notion of separating system from Hamming space to Euclidean space,
and how Erd\"os and F\"uredi went precisely the opposite way 35 years ago.
It could be desirable to start a more systematic study of separating systems in various metric spaces.
A first candidate for this is real $L^1$ space, that is, $\R^n$ equipped with the norm $\norme{x}_{L^1}=\abs{\pi_1(x)}+\cdots+\abs{\pi_n(x)}$.
It is easy to show that this metric space satisfies analogues of our (P1) (P2) (P3), the precise formulation of which we leave as an exercise to the reader.
In particular we see that there again, segments are convex;
actually the segment from $x$ to $y$ in this space is the \emph{standard box} $[x,y]=\{z\in\R^n\;|\;\forall i,\, \pi_i(x)\leq\pi_i(z)\leq\pi_i(y)\textrm{ or }\pi_i(y)\leq\pi_i(z)\leq\pi_i(x)\}$.
From this description we see that $(2,1)$-separation in real $L^1$ space is precisely the problem studied by Alon, F\"uredi and Katchalski in \cite{AFK1985}!
Interestingly, these authors use the term ``separation'', although they do not refer to the literature on separating systems.
Moreover they point out that their work borrows inspiration from \cite{ES1935}, which itself is loosely related to \cite{ErdosFuredi}.
It is remarkable that all these incarnations of separating systems in various metric spaces have such connections with combinatorial geometry.
Certainly this phenomenon would deserve further study.


\begin{thebibliography}{1}

\bibitem[AB09]{AB2009}
E.~Ackerman \& O.~Ben-Zwi.
On sets of points that determine only acute angles.
Europ. J. Combin. 30 (2009) 908--910.

\bibitem[AZ]{AignerZiegler}
M.~Aigner \& G.~Ziegler.
Proofs from THE BOOK.
Fifth edition.
Springer-Verlag, Berlin, 2014.

\bibitem[AFK85]{AFK1985}
N.~Alon, Z.~F\"uredi \& M.~Katchalski.
Separating pairs of points by standard boxes.
Europ. J. Combin. 6 (1985) 205--210.

\bibitem[AS]{AlonSpencer}
N.~Alon \& J.~Spencer.
The probabilistic method.
Fourth edition.
Wiley Series in Discrete Mathematics and Optimization, 2016

\bibitem[Bev06]{Bevan2006}
D.~Bevan.
Sets of points determining only acute angles and some related colouring problems.
Electron. J. Combin. 13 (2006).

\bibitem[CCX11]{CCX2011}
I.~Cascudo, R.~Cramer \& C.~Xing.
The torsion-limit for algebraic function fields and its application to arithmetic secret sharing,
pp.~685-705 of:
Advances in Cryptology --- CRYPTO 2011.
Lecture Notes in Computer Science 6841, Springer, 2011.

\bibitem[Coh09]{Cohen2009}
G.~Cohen.
Separation and witnesses,
pp.~12--21 of:
Coding and Cryptology, Second International Workshop, IWCC 2009.
Lecture Notes in Computer Science 5557, Springer, 2009.

\bibitem[CL85]{CL1985}
G.~Cohen \& A.~Lempel.
Linear intersecting codes.
Discr. Math. 56 (1985) 35--43.

\bibitem[CS03]{CS2003}
G.~Cohen \& H.~G.~Schaathun.
Asymptotic overview on separating codes.
Reports in Informatics 248, Univ.~Bergen, 2003.

\bibitem[EFF82]{EFF1982}
P.~Erd\"os, P.~Frankl \& Z.~F\"uredi.
Families of finite sets in which no set is covered by the union of two others.
J. Combin. Theory Ser. A 33 (1982) 158--166.

\bibitem[EF83]{ErdosFuredi}
P.~Erd\"os \& Z.~F\"uredi.
The greatest angle among $n$ points in the $d$-dimensional Euclidean space,
pp.~275--283 of:
Combinatorial mathematics (Marseille-Luminy, 1981),
Annals of Discrete Math. 17,
North-Holland Math. Stud. 75,
North-Holland, Amsterdam, 1983.

\bibitem[ES35]{ES1935}
P.~Erd\"os \& G.~Szekeres.
A combinatorial problem in geometry.
Compositio Math. 2 (1935) 463--470.

\bibitem[FK84]{FK84}
M.~L.~Fredman \& J.~Koml\'os.
On the size of separating systems and families of perfect hash functions.
SIAM J. Alg. Disc. Meth. 5 (1984) 538--544.

\bibitem[FGU69]{FGU69}
A.~D.~Friedman, R.~L.~Graham \& J.~D.~Ullman.
Universal single transition time asynchronous state assignments.
IEEE Trans. Comput. 18 (1969) 541--547.

\bibitem[K\"or95]{Korner1995}
J.~K\"orner.
On the extremal combinatorics of the Hamming space.
J.~Combin. Theory Ser. A  71  (1995) 112--126.

\bibitem[OEIS]{OEIS}
N.~Sloane.
On-Line Encyclopedia of Integer Sequences.
Sequence A089676, \url{https://oeis.org/A089676}

\bibitem[Ran01]{CRAS}
H.~Randriambololona.
Comportement asymptotique des hauteurs des sous-sch\'emas de dimension nulle de l'espace projectif.
C. R. Acad. Sci. Paris S\'er. I Math. 333 (2001) 329--332.

\bibitem[Ran03]{NTC}
H.~Randriam.
Autour du crit\`ere de Xing pour les codes s\'eparants.
Report, RTP 13 AS ``Nouvelles tendances en cryptographie'', Sept. 2003.

\bibitem[Ran10]{ITW2010}
H.~Randriam.
Hecke operators with odd determinant and binary frameproof codes beyond the probabilistic bound?
Proc. 2010 IEEE Information Theory Workshop (ITW 2010 Dublin).

\bibitem[Ran11]{2D-G}
H.~Randriambololona.
Diviseurs de la forme $2D-G$ sans sections et rang de la multiplication dans les corps finis.
Preprint, \url{http://arxiv.org/abs/1103.4335}

\bibitem[Ran12]{ChCh}
H.~Randriambololona.
Bilinear complexity of algebras and the Chudnovsky-Chudnovsky interpolation method.
J. Complexity 28 (2012) 489--517.

\bibitem[Ran13]{IJM}
H.~Randriambololona.
$(2,1)$-separating systems beyond the probabilistic bound.
Israel J. Math. 195 (2013) 171--186. 

\bibitem[Ren61]{Renyi1961}
A.~R\'enyi.
On random generating elements of a finite Boolean algebra.
Acta Sci. Math. Szeged 22 (1961) 75--81.

\bibitem[Sag78]{Saga1978}
Yu.~Sagalovich.
Cascade codes of automata states.
Probl. Peredachi Inf. 14 (1978) 77--85.

\bibitem[SC09]{SC2009}
Yu.~Sagalovich \& A.~Chilingarjan.
Separating systems and new scopes of its application.
Information Processes 9 (2009) 225--248.

\bibitem[Ser97]{Serre1997}
J.-P.~Serre.
R\'epartition asymptotique des valeurs propres de l'op\'erateur de Hecke $T_p$.
J.~Amer. Math. Soc. 10 (1997) 75--102.

\bibitem[TV95]{TV1995}
M.~Tsfasman \& S.~Vladut.
Geometric approach to higher weights.
IEEE Trans. Inform. Theory 41 (1995) 1564--1588.

\bibitem[Xin02]{Xing2002}
C.~Xing.
Asymptotic bounds on frameproof codes.
IEEE Trans. Inform. Theory 48 (2002) 2991--2995.

\bibitem[Wei91]{Wei1991}
V.~Wei.
Generalized Hamming weights for linear codes.
IEEE Trans. Inform. Theory 37 (1991) 1412--1418. 

\end{thebibliography}
\end{document}